\documentclass[12pt]{amsart}
\usepackage{amssymb}
\usepackage{a4wide}
\usepackage{color}
 \usepackage{boxedminipage}
\usepackage{enumerate, graphicx}

\def\KK{{\mathbb K}}
\def\RR{{\mathbb R}}
\def\PP{\mathcal P}
\def\QQ{\mathcal Q}
\def\SS{\mathcal S}
\def\Var{{\mathfrak V}}
\def\tT{{\mathcal T}}
\def\lL{{\mathcal L}}
\def\kK{{\mathcal K}}

\def\mM{{\mathcal M}}
\def\nN{{\mathcal N}}
\def\Id{{\mathcal I}}
\def\Mat{{\mathrm{Mat}}}
\def\rank{{\mathrm{rank}}}
\def\cstar{{\mathrm{star}}}
\def\nnull{{\mathbf{0}}}
\DeclareMathOperator\Sep{{\mathrm{Sep}}}

  \newenvironment{note}[1][Note]
   {\bigskip\begin{center}\begin{boxedminipage}{4.5in}\setlength{\parindent}{1em}\noindent\textbf{#1. }}
   {\end{boxedminipage}\end{center}\bigskip}

  \makeatletter
  \CheckCommand*\refstepcounter[1]{\stepcounter{#1}%
      \protected@edef\@currentlabel
       {\csname p@#1\endcsname\csname the#1\endcsname}%
  }
  \renewcommand*\refstepcounter[1]{\stepcounter{#1}%
    \protected@edef\@currentlabel
      {\csname p@#1\expandafter\endcsname\csname the#1\endcsname}%
  }
  \def\labelformat#1{\expandafter\def\csname p@#1\endcsname##1}
  \DeclareRobustCommand\Ref[1]{\protected@edef\@tempa{\ref{#1}}%
     \expandafter\MakeUppercase\@tempa
  }
  \makeatother

  \makeatletter
  \newcommand{\numberlike}[2]{%
     \expandafter\def\csname c@#1\endcsname{%
         \expandafter\csname c@#2\endcsname}%
  }
  \makeatother

  \def\DefaultNumberTheoremWithin{section}

  \theoremstyle{plain}
  \newtheorem{Lemma}{Lemma}
     \numberwithin{Lemma}{\DefaultNumberTheoremWithin}
     \labelformat{Lemma}{Lemma~#1}
  
     \numberwithin{Claim}{\DefaultNumberTheoremWithin}
     \numberlike{Claim}{Lemma}
     \labelformat{Claim}{Claim~#1}
  \newtheorem{Theorem}{Theorem}
     \numberwithin{Theorem}{\DefaultNumberTheoremWithin}
     \numberlike{Theorem}{Lemma}
     \labelformat{Theorem}{Theorem~#1}
  \newtheorem{Corollary}{Corollary}
     \numberwithin{Corollary}{\DefaultNumberTheoremWithin}
     \numberlike{Corollary}{Lemma}
     \labelformat{Corollary}{Corollary~#1}
  \newtheorem{Proposition}{Proposition}
     \numberwithin{Proposition}{\DefaultNumberTheoremWithin}
     \numberlike{Proposition}{Lemma}
     \labelformat{Proposition}{Proposition~#1}
  
     \numberwithin{Conjecture}{\DefaultNumberTheoremWithin}
     \numberlike{Conjecture}{Lemma}
     \labelformat{Conjecture}{Conjecture~#1}
  \newtheorem{Situation}{Situation}
     \numberwithin{Situation}{\DefaultNumberTheoremWithin}
     \numberlike{Situation}{Lemma}
     \labelformat{Situation}{Situation~#1}

  \theoremstyle{definition}
  
     \numberwithin{Definition}{\DefaultNumberTheoremWithin}
     \numberlike{Definition}{Lemma}
     \labelformat{Definition}{Definition~#1}

  \theoremstyle{definition}
  
     \numberwithin{Question}{\DefaultNumberTheoremWithin}
     \numberlike{Question}{Lemma}
     \labelformat{Question}{Question~#1}

  \theoremstyle{definition}
  
     \numberwithin{Problem}{\DefaultNumberTheoremWithin}
     \numberlike{Problem}{Lemma}
     \labelformat{Problem}{Problem~#1}

     \theoremstyle{remark} \newtheorem{Remark}{Remark}
     \numberwithin{Remark}{\DefaultNumberTheoremWithin}
     \numberlike{Remark}{Lemma}
     \labelformat{Remark}{Remark~#1}
  \theoremstyle{remark}

     \numberwithin{Example}{\DefaultNumberTheoremWithin}
     \numberlike{Example}{Lemma}
     \labelformat{Example}{Example~#1}
  
     \labelformat{Case}{Case~#1}
     \numberwithin{Case}{Lemma}
  
     \labelformat{Step}{Step~#1}
     \numberwithin{Step}{Lemma}

\labelformat{section}{Section~#1}
  \labelformat{subsection}{Section~#1}
  \labelformat{subsubsection}{Section~#1}

\title{The Varchenko Determinant for Oriented Matroids}

\author{Winfried Hochst\"attler}
    \address{FernUniversit\"at in Hagen \\ 
          Fakult\"at f\"ur Mathematik und Informatik \\
          58084 Hagen\\
          Germany}
     \email{Winfried.Hochstaettler@fernuni-hagen.de}

 \author{Volkmar Welker}
     \address{Fachbereich Mathematik und Informatik\\
              Philipps-Universit\"at Marburg\\
              35032 Marburg\\
              Germany}
     \email{welker@mathematik.uni-marburg.de}

\begin{document}
\begin{abstract}
  We generalize the Varchenko matrix of a hyperplane arrangement to
  oriented matroids. We show that the celebrated determinant formula
  for the Varchenko matrix, first proved by Varchenko, generalizes to
  oriented matroids. It follows that the determinant only depends on
  the matroid underlying the oriented matroid and analogous formulas
  hold for closed supertopes in oriented matroids.  We follow a proof 
	strategy for
  the original Varchenko formula first suggested by Denham and Hanlon.
  Besides several technical lemmas this strategy also requires a
  topological result on supertopes which is of independent interest.
  We show that a supertope considered as a subposet of the tope poset
  has a contractible order complex.
\end{abstract}

\maketitle

\section{Introduction}
\label{sec:intro}

Let $\lL$ be an oriented matroid on a finite ground set $E$ given as 
a set of covectors $X = (X_e)_{e \in E} \in \{+,-,0\}^E$.
We denote by $\tT = \tT(\lL)$ the set of topes in 
$\lL$ and call for two topes $P = (P_e)_{e \in E}$ and 
$Q = (Q_e)_{e \in E}$ the set $\Sep(P,Q) = \{e\in E~|~P_e = -Q_e \ne 0\}$ 
the \emph{separator} of $P$ and $Q$. 
For the oriented matroid $\lL$ and a field $\KK$
we consider the polynomial ring $\KK[U_e ~|~e \in E]$ in the set of variables
$U_e$, $e \in E$.
We call the following matrix $\Var = \Var(\lL)$ the \emph{Varchenko matrix} of $\lL$.
The matrix $\Var$ is the
$(\# \tT \times \# \tT)$-matrix over $\KK[U_e ~|~e \in E]$ with
rows and columns indexed by the topes $\tT$ in a fixed linear order. 
For $P ,Q \in \tT$ the
entry $\Var_{P,Q}$ in row $P$ and column $Q$ is given by
$\displaystyle{\prod_{e \in \Sep(P,Q)} U_e}$.
In particular, all entries $\Var_{P,P}$ on the diagonal are equal 
to $1$.
For $F \in \lL$ we set $a(F) := \displaystyle{\prod_{\genfrac{}{}{0pt}{}{e \in E}{F_e = 0}} U_e}$. 
In this paper we prove:

\begin{Theorem}
   \label{thm:varchenko}
   Let $\Var$ be the Varchenko matrix of the oriented matroid with covector set $\lL$.
   Then
   \begin{eqnarray*}
     \det (\Var) & = & \prod_{F \in \lL} (1-a(F)^2)^{b_F}. 
   \end{eqnarray*}
   for nonnegative integers $b_F$. 
\end{Theorem}

Note, that a factor $(1-a(F)^2)$ is zero if and only if $F$ is a tope. 
In this case it turns out that $b_F = 0$. By the convention $0^0 = 1$
it follows that $\det(\Var)
\neq 0$.  In \ref{cor:varchenko} we give an alternative formulation of the 
product formula which will shed more light on the exponents $b_F$. 
In particular, it will follow that $\det
(\Var)$ only depends on the matroid underlying the oriented matroid
defined by $\lL$.

If $\lL$ is given as the set of covectors of a hyperplane arrangement in some 
$\RR^n$ then $\Var$ is the Varchenko matrix of the hyperplane arrangement
and \ref{thm:varchenko} is Varchenko's result from \cite{V}. 
Initially, Varchenko was motivated by the case of the reflection
arrangements of the symmetric group. In that case the matrix relates to 
Drinfeld--Jimbo quantized Kac–Moody Lie algebras in type $A$. 
This relation had been unraveled by Schechtman and Varchenko in \cite{SV2}. There it is shown that the
kernel of specializations of the matrix describe the Serre relations for the
algebra. Motivated by these facts Zagier \cite{Z} gave a proof of the
determinant formula for the reflection arrangement of the symmetric group 
based on calculations in its group algebra. 
Work of Hanlon and Stanley \cite{HS} ties in the matrix and its kernel 
with combinatorial aspects of the representation theory of the symmetric group
when all variables are substituted by a fixed complex number.
For general arrangements of hyperplanes, Denham and Hanlon \cite{DH} show that the matrix and its determinant
can be used in an approach to determine the Betti numbers of the Milnor fiber
of the complexified arrangements; that is the fiber in complex space of the product of linear
forms defining the hyperplanes at complex numbers different from $0$.

After the original proof in \cite{V} there were attempts in \cite{DH} and \cite{G}
to provide a cleaner proof of Varchenko's original result.  Our
approach generalizes ideas from \cite{DH} and \cite{G} to oriented
matroids and replaces the problematic parts from both works by
alternative arguments. 
Recently, a new proof using a different
strategy was published in \cite{AM}. We have not studied this proof
thoroughly and cannot judge if it generalizes to oriented matroids as
well. This paper is not the first to study oriented matroid
generalizations of the Varchenko determinant formula. In the works
\cite{Ve1,Ve2} an approach is sketched for proving \ref{thm:varchenko}
originally for general oriented matroids in \cite{Ve1} and restricted
to oriented matroids that allow a representation as a pseudo point
configuration, only, in the subsequent \cite{Ve2}.  Despite several
attempts we were not able to follow the argumentation of either
thesis. Philosophically, our work parallels the article of Bry\l awski and Varchenko
\cite{BV} who give a matroid generalization of a determinant formula by 
Schechtman and Varchenko \cite{SV1} for yet another important class of matrices arising in 
representation theory. That paper probably also motivated \cite{Ve1} and
\cite{Ve2}. 

Besides amendments and the generalization to oriented matroids the
key new ingredient in our proof of \ref{thm:varchenko} is the following result
which we consider of independent interest.
For its formulation, let $R \in \tT$ be a fixed base tope and consider $\tT$ as
a partially ordered set with order relation $P \preceq_R Q$ if $\Sep(R,P)
\subseteq \Sep(R,Q)$. We write $\tT_R$ if we consider $\tT$ with this
partial order. 
For disjoint subsets $S^+, S^- \subseteq E$ such that $S^+ \cup S^-\ne
\emptyset$ the set of topes
\[\tT(S^+,S^-):=\left\{T \in \tT \mid T_f = + \text{ for all }f \in S^+\text{ and }T_f = -
  \text{ for all }f \in S^-\right\}\] is called a \emph{supertope}.  By
\cite[Proposition 4.2.6]{thebook} supertopes are exactly the \emph{$T$-convex sets}, i.e.\ the sets of topes that contain any shortest
path between any of two of its members.  We call a supertope
$\tT(S^+,S^-)$ a \emph{closed supertope}, if for all supertopes $\tT(\tilde
S^+,\tilde S^-)$ such that $S^+ \subseteq\tilde S^+, S^-
\subseteq\tilde S^- $ but $(S^+, S^-) \ne (\tilde S^+, \tilde S^-)$
necessarily $\tT(\tilde
S^+,\tilde S^-) \subsetneq \tT(S^+, S^-)$.
In case the oriented matroid is given by an arrangement of hyperplanes then 
a closed supertope corresponds to a closed cone cut out by the hyperplanes
from the arrangement.
Note that our notion of closed supertope is more general than the notion of a cone  from \cite[Definition 10.1.1 (iii)]{thebook}.

One would expect that
$T$-convex sets as subsets of the tope poset $\tT_R$ are contractible.
We will show that this is indeed the case.

\begin{Theorem}
  \label{thm:supertope}
  Let $R \in \tT$ be the base tope of the poset $\tT_R$ and
  $\tT(S^+,S^-)\ne \emptyset$ be a supertope. Then 
  $\tT(S^+,S^-)$ considered as subposet of
  $\tT_R$ is contractible.
\end{Theorem}

The paper is organized as follows. In \ref{sec:basic} we recall some
basic notations and results from oriented matroid theory and poset
topology. We then use tools from poset topology to derive
results on the topology of complexes associated to oriented matroids
in \ref{sec:omtopology}.  In \ref{sec:supertope} we provide the proof
of \ref{thm:supertope} and exhibit why we cannot follow the
argumentation from \cite{DH} and \cite{G}. In \ref{sec:varchenko} we
prove \ref{thm:varchenko}. The key step in the proof is a
factorization of the Varchenko matrix, one factor for each element of
the ground set $E$ (Proposition 5.3).  The key ingredient of the
factorization is a result on M\"obius numbers which is a direct
consequence of \ref{thm:supertope} (Corollary 4.5).  Then the
determinant of each factor is analyzed.  M\"obius number implications
of topological results from \ref{sec:omtopology} then show that each
is block upper triangular with controllable block structure (Lemma
5.6). Now \ref{thm:varchenko} follows via basic linear algebra.  As a
corollary we give a description of the numbers $b_F$ which implies
that the determinant only depends on the matroid underlying the
oriented matroid. As a second corollary we show that the result
extends to closed supertopes 
and hence in particular to affine oriented matroids.
 
\section{Background on Oriented Matroids and Poset Topology}
\label{sec:basic}

\subsection{Poset Topology} In this paper we will associate various
partially ordered sets, posets for short, to oriented matroids. For
our purposes it turns out to be useful to consider a poset $\PP$ as a
topological space.  We do this by identifying $\PP$ with its order
complex, respectively the geometric realization of the order complex.
Recall that the \emph{order complex} of a poset $P$ is the simplicial
complex whose chains are the linearly ordered subsets of $\PP$.  Using
this identification we can speak about contractible and homotopy
equivalent posets.  We will employ the following standard tools from
poset topology (see \cite{Bj} for details).  For their formulation we
denote for a poset $\PP$ and $p \in \PP$ by $\PP_{\leq p}$ the
subposet $\{ q \in \PP~|~q \leq p\}$. Analogously defined are $\PP_{<
  p}$,$\PP_{>p}$ and $\PP_{\geq p}$. For $p \leq q$ in $\PP$ we write
$(p,q)_\PP$ for the \emph{open interval} $\PP_{> p} \cap \PP_{< q}$
and $[p,q]_\PP$ for the \emph{closed} interval $\PP_{ \geq p} \cap
\PP_{\leq q}$.

\begin{Proposition}[Quillen Fiber Lemma]
   \label{lem:quillen}
   Let $\PP$ and $\QQ$ be posets and $f : \PP \rightarrow \QQ$ a poset map.
   If for all $q \in \QQ$ we have that
   $f^{-1} (\QQ_{\leq q})$ is contractible, then $\PP$ and $\QQ$ are homotopy
   equivalent.
\end{Proposition}
 
By simple induction on $\#\SS$ one derives the following corollary.
 
\begin{Corollary}
  \label{cor:subquillen}
  Let $\PP$ be a poset and $\SS$ a subset such that $\PP_{< s}$ is
  contractible for all $s \in \SS$. Then $\PP\setminus
    \SS$ and $\PP$ are homotopy equivalent.
\end{Corollary}

We will use poset topology also to prove results on the M\"obius number of
a poset $\PP$. 
For that we take advantage of the following well known numerical consequence of
the fact that two posets are homotopy equivalent. 
For a poset $\PP$ we denote by $\mu(\PP)$ the \emph{M\"obius number} of $\PP$
(see \cite[Chapter 3]{St}).  

\begin{Proposition}
  \label{pr:moebius}
  For two homotopy equivalent posets $\PP$ and $\QQ$ we have
  $\mu(\PP) = \mu(\QQ)$. In particular, if $\PP$ is contractible then
  $\mu(\PP) = 0$. 
\end{Proposition}

\subsection{Oriented Matroids} 

As mentioned in \ref{sec:intro} we consider an oriented matroid $\lL$ on ground set $E$ 
as a set of covectors $X = (X_e)_{e \in E}\in \{+,-,0\}^E$. In our notation we follow \cite{thebook} which
also contains all required background information on oriented matroids. Frequently, we will
use the following definitions and notations.

We order the covectors by the
product order induced by the order $0 < +,-$ and write $\nnull =
(0)_{e \in E}$ for the unique minimal covector in this order.  Following our
conventions, for a covector $X$ we write $(\nnull,X)_\lL$ for the open interval from 
$\nnull$ to $X$ in $\lL$. 
It is well known that the poset of
covectors is graded and hence one can assign each covector $X \in
\lL$ a rank $\rank_\lL(X)$.  The rank $\rank(\lL)$ of $\lL$ is defined
as the maximal rank of one of its covectors.

As usual for a covector $X \in \lL$ we write $X^+$ for 
$\{ e \in E~|~X_e = +\}$ and $X^-$ for $\{ e \in E~|~X_e = -\}$.
In addition, we write $z(X) = \{ e \in E~|~X_e = 0\}$ for its
zero-set. 

Let $A \subseteq E$ be a nonempty set. For $F \in \lL$ we denote by
$F|_A$ the covector $(F_e)_{e \in A}$.  For a set $\kK$ of covectors
over $E$ we then write $\kK|_A$ for the set $\{ F|_A~|~F \in \kK \}$
of covectors over $A$.  For an oriented matroid $\lL$ over $E$ and a
nonempty subset $A \subseteq E$ the set of covectors $\lL|_A$ defines
an oriented matroid called the restriction of  $\lL$ to $A$.  The
contraction of $A$ in $\lL$ is the oriented matroid $\lL / A$ with
covector set $\{ F|_{E\setminus A}~|~F \in \lL, z(F) \subseteq A \}$.
In case $A = \{f\}$ is a singleton we also write $\lL / f$ for $\lL /
A$.

For two covectors $X, Y \in \lL$ their \emph{composition} $X \circ Y$ is
defined by $(X \circ Y)^+ = X^+ \cup (Y^+ \setminus X^-)$ and
$(X \circ Y)^- = X^- \cup (Y^- \setminus X^+)$.

Next we repeat and extend some notation already stated in \ref{sec:intro}.
We write $\tT(\lL)$ for the set of topes of $\lL$ and simply $\tT$ in
case there is no danger of ambiguity. 
For $P,Q \in \tT$ we denote by $\Sep(P,Q)$ the \emph{separator}
of $P$ and $Q$. Then for fixed $R \in \tT$ the set of topes $\tT$
carries a poset structure defined by $P \preceq_R Q$ if and only if
$\Sep(R,P) \subseteq \Sep(R,Q)$ (see \cite[Definition 4.2.9]{thebook}).
We write $\tT_R$ to denote $\tT$ with this poset structure. 
In order to reduce the number of double subscripts we write
write $(P,Q)_R$ for $(P,Q)_{\tT_R}$ 
and $[P,Q]_R$ for $[P,Q]_{\tT_R}$.
For $e \in E$ and $P \in \tT$ we say that 
$e$ \emph{does not define a proper face} of $P$
if the only covector $F \in \lL$ with $F \leq P$ and $F_e =
0$ is $F = \nnull$.  

We will frequently encounter the situation where $R,P \in \tT$ 
and $e \in E$ are such that $+ = R_e$ and $- = P_e$. Then after 
reordering and reorientation we can assume the following.

\begin{Situation}
  \label{sit:*} 
     $R =+ \cdots +$ and $P=- \cdots -+\cdots +$ and $e$ is the first 
coordinate of our sign vectors. 
\end{Situation}

In the rest of the paper, we will work in the general situation unless
there is a technical simplification when assuming \ref{sit:*}. In that case
we will explicitly mention the assumption.

Next we state well known facts about the topology of $\lL$ and $\tT_R$. 
 
\begin{Lemma}[Lemma 4.3.11 \cite{thebook}]
  \label{lem:contr}
  Let $P,R \in \tT$ set 
  $$F_R(P) = \{ X \in (\nnull,P)_\lL~|~z(X) \subseteq \Sep(P,R) \} \subseteq \lL.$$
  Then $F_R(P)$ is a filter in $(\nnull,P)_\lL$. If $P \neq \pm R$ then 
  $F_R(P)$ is contractible.
\end{Lemma}

%
%

For a covector $X$ we set $\cstar(X) := \{ T \in \tT~|~X \leq T\}$.

\begin{Theorem}[Theorem 4.4.2 \cite{thebook}] 
  \label{the:edel}
  Let $\lL$ be an oriented matroid of rank $r$ and $R \in \tT$. 
  For $T_1,T_2 \in \tT_R$ such that $T_1 \preceq_R T_2$
  the order complex of $(T_1,T_2)_R$ 
  is homotopy equivalent to
  \begin{enumerate}[(i)]
    \item a sphere of  dimension $r-\rank_\lL(X)-2$ if 
       $[T_1,T_2]_R$ equals $\cstar(X)$ for some covector $X$,
    \item a point, i.e. it is contractible, otherwise.
  \end{enumerate}
\end{Theorem}

For $e \in E$ and $R \in \tT$ we write $\tT_{R,e}$ for the poset 
$\{ T \in \tT~|~T_e = -R_e\}\cup \{ \hat{0}\}$ 
with $\hat{0}$ as its least element and the remaining poset 
structure induced from $\tT_R$.
For $P \in \tT_{R,e}$ we write $(\hat{0},P)_{R,e}$ for the interval from
$\hat{0}$ to $P$ in $\tT_{R,e}$. 
We set $$(\hat{0},P)_{R,e}^{\triangle} := 
\{ Q \in (\hat{0},P)_{R,e} ~|~\exists X \in \lL ~:~ [Q,P]_R = \cstar(X) \}.$$
The following is an immediate consequence of \ref{the:edel}.

\begin{Corollary}
  \label{cor:smallposet}
  Let $R \in \tT$, $e \in E$ and $P \in \tT_{R,e}$.
  Then $(\hat{0},P)_{R,e}$ and $(\hat{0},P)_{R,e}^\triangle$ 
  are homotopy equivalent.
\end{Corollary}
\begin{proof}
  Let 
  $$S= \{ Q \in (\hat{0},P)_{R,e} ~|~\nexists X \in \lL ~:~ [Q,P]_R = \cstar(X) \}.$$
  Then $(\hat{0},P)_{R,e}^\triangle = (\hat{0},P)_{R,e} \setminus S$.
  For $Q \in S$ \ref{the:edel} implies that 
  $((\hat{0},P)_{R,e})_{> Q} = (Q,P)_{R}$ is contractible.
  Now the assertion follows from \ref{cor:subquillen}.
\end{proof}
 
\section{Some Oriented Matroid Topology}
\label{sec:omtopology}

At the end of the last section we already recalled some known facts about the topology of posets associated 
to  oriented matroids. This section now contains oriented matroid generalizations of topological
results stated in \cite{DH} and \cite{G} for hyperplane arrangements. 
 
For $P \in \tT_{R,e}$ we set $S = E \setminus \Sep(P,R)$ and 
$S' = \Sep(P,R) \setminus \{e\}$. 
Let $B \in \tT(\lL|_{S'})$ be the unique tope 
from $\tT(\lL|_{S'})$ such that
$B_f = P_f = -R_f$ for all $f \in S'$. 
Let $G \in \tT(\lL|_S)$ be the unique tope from $\tT(\lL|_S)$ such that
$G_f = P_f$ for all $f \in S$. 

We define 
$$W_{R,e}(P) = \left\{ F \in (\nnull,P)_\lL \mid F_e = -R_e,\,  F|_S = G,\, F|_{S'} \leq B  \right\} \subseteq \lL.$$

\noindent We consider $W_{R,e}(P)$ as a poset with order relation inherited from $\lL$. 
Assuming \ref{sit:*} we have:
$$W_{R,e}(P) = \left\{ F \in (\nnull,P)_\lL \mid F_e = -,\,  F|_S = \{+\}^S,\, F|_{S'} \leq \{-\}^{S'}  \right\} \subseteq \lL.$$

\noindent We consider the following map:

$$\alpha_P : \left\{ \begin{array}{ccc} (\hat{0},P)_{R,e}^\triangle & \rightarrow & \lL \\
                                         C & \mapsto & \alpha_P(C) := X, {\genfrac{}{}{0pt}{}{\tiny{\mbox{~for the~}} X \in \lL}{\tiny{\mbox{~such that~}} (C,P)_R = \cstar(X)}} 
                    \end{array}
           \right.
$$

\begin{Lemma} 
   \label{lem:posetmap}
   Let $R \in \tT$ and $e \in E$.
   For $P \in \tT_{R,e}$ and $C \in (\hat{0},P)_{R,e}^\triangle$ 
   the following holds:
   \begin{enumerate}[(i)]     \item $z(\alpha_P(C))=\Sep(C,P)$ and
   $\alpha_P(C) \in W_{R,e}(P)$.
     \item $\alpha_P$ is a poset map from 
        $(\hat{0},P)_{R,e}^\triangle$ to $W_{R,e}(P)$. 
      \item For $F \in W_{R,e}(P)$ the tope $F\circ R$ is the unique
        maximal element in $\alpha_P^{-1} (W_{R,e}(P)_{\leq F})$.
    \end{enumerate}
\end{Lemma}
\begin{proof}
  We assume \ref{sit:*}.
  \begin{enumerate}[(i)]
  \item By definition $[C,P]_R = \cstar(X)=[X \circ R, X
    \circ(-R)]_R$. Hence, $X \le P$ and $X_e=C_e=P_e=-$, implying
    $z(X)=\Sep(C,P)$ and $\alpha_P(C) =X \in W_{R,e}(P)$.
   \item
      Since $C \preceq_R C'$ in $(\hat{0},P)_{R,e}^\triangle$ implies
      $\Sep(C',P) \subseteq \Sep(C,P)$ it follows from (i) that 
      $\alpha_P(C) \leq \alpha_P(C')$. Hence $\alpha_P$ is a map of posets.  
    \item Let $F,F' \in W_{R,e}(P)$, $F' \le F$ and $C=F
      \circ R$. We have $C' \in \alpha_P^{-1}(F')$ if and only if
      $[C',P]_R=[F'\circ R, F'\circ (-R)]_R$. As $z(F) \subseteq
      z(F')$ this implies $\Sep(C',R) \subseteq \Sep(C,R)$ and hence
      the assertion.

  \end{enumerate}
\end{proof}

The following proposition allows us to determine the topology of 
the posets $(\hat{0},P)_{R,e}$ through known results on $W_{R,e}(P)$.

\begin{Proposition}
  \label{pro:homotopy}
  Let $R \in \tT$ and $e \in E$ such that $R_e = +$ and
  $P \in \tT_{R,e}$. Then:
  \begin{enumerate}[(i)]
    \item The order complex of the interval 
        $(\hat{0},P)_{R,e}$ and the order complex of $W_{R,e}(P)$ are 
        homotopy equivalent. 
   \item For $e$ that do not define 
        a proper face of $P$ the order complex of $W_{R,e}(P)$ is 
        contractible if $\pm R \neq P$ and homotopy equivalent
        to a $(\rank(\lL)-2)$-sphere if $-R = P$.
  \end{enumerate}
\end{Proposition}
\begin{proof}
  \begin{enumerate}[(i)]
    \item 
  From \ref{cor:smallposet} it follows that 
  $(\hat{0},P)_{R,e}$ and $(\hat{0},P)_{R,e}^\triangle$ are 
  homotopy equivalent.

  Using \ref{lem:posetmap} (iii) it follows that the order complex of each 
  fiber $\alpha_P^{-1} (W_{R.e}(P)_{\leq F})$ for $F \in W_{R,e}(P)$ is a 
  cone and 
  hence contractible. Now the Quillen Fiber Lemma, \ref{lem:quillen}, 
  shows that 
  the order complexes of $(\hat{0},P)_{R,e}^\triangle$ and $W_{R,e}(P)$ are 
  homotopy equivalent.
    \item 
  Since $e$ does not define a proper face of $P$, we have
  $W_{R,e}(P) = F_R(P)$. 
  If $P = -R$ then $W_{R,e}(P) = (\nnull,P)_\lL$ and hence is
  homotopy equivalent to a $(\rank(\lL)-2)$-sphere.
  If $P \neq \pm -R$ then $W_{R,e}(P) = F_R(P)$ and \ref{lem:contr}
  shows that $W_{R,e}(P)$ is contractible.
  \end{enumerate}
\end{proof}

We summarize the results in the following theorem.

\begin{Theorem}
  \label{thm:topmoebius}
  Let $P \in \tT_{R,e}$ such that $e$ does not define a proper face of $P$.
  Then the interval 
  $(\hat{0},P)_{R,e}$ is  
  contractible if $-R \neq P$ and homotopy equivalent to a 
  $(\rank(\lL)-2)$-sphere if $-R = P$. 
\end{Theorem}
\begin{proof}
  The result is an immediate consequence of \ref{pro:homotopy}(i)
  and (ii).
\end{proof}

The well known connection of homotopy type and
M\"obius-number from \ref{pr:moebius} yields.

\begin{Corollary}
  \label{co:moebius}
  Let $P \in \tT_{R,e}$ such that $e$ does not define a proper face of $P$.
  Then the M\"obius 
  number 
  $\mu((\hat{0},P)_{R,e})$ is $0$ if $-R \neq P$ and   
  $(-1)^{\rank(\lL)}$ if $-R = P$. 
\end{Corollary}

The next result overlaps with \ref{thm:topmoebius} but also covers some of the cases where
$e$ defines a face of $P$. Note that $e$ defines a proper face of $R$ if and only if $e$ defines
a proper face of $-R$. Hence the interval $(\hat{0},-R)_{R,e}$ is covered by \ref{thm:topmoebius} if $e$ does not
define a proper face of $R$ and it is covered by \ref{thm:moebius}
otherwise. 


\medskip

\begin{Theorem}
  \label{thm:moebius}
  Let $R \in \tT$ and let $e \in E$ define a proper face of $R$. 
  Let $F \in \lL$ 
  be the maximal covector such that $F \leq R$ and $F_e = 0$
  and choose 
  $P_{top} \in \tT_{R,e} \setminus \cstar(F)$. Then
  $(\hat{0},P_{top})_{R,e}$ is contractible. 
  In particular, $\mu((\hat{0},P_{top})_{R,e}) = 0$.
\end{Theorem}
\begin{proof}
  Let $P \in (\hat{0},P_{top})_{R,e}$. Then by the gate property
  \cite[Exercise 4.10]{thebook} the tope
  $Q = F \circ P \in \cstar(F)$ is the unique tope in $\cstar(F)$ 
  such that for all $O \in \cstar(F)$ we have
  \begin{eqnarray*}
    \Sep(P,O) & = & \Sep(P,Q) \cup \Sep(Q,O) \\
    \emptyset & = &  \Sep(P,Q) \cap \Sep(Q,O).
  \end{eqnarray*}
  Since $F_e = 0$ it also follows that $Q_e = -$.
  Since $F \le R$, clearly $\Sep(R,Q) =\Sep(R,F \circ P)\subseteq
  \Sep(R,P)$ and hence $Q \preceq_R P$.  This shows $Q \in
  (\hat{0},P_{top})_{R,e}$.  Now let $Q \leq_R Q'$. Then $F \circ Q
  \preceq_R F \circ Q'$. Since $F \leq R$ it follows that $F \circ Q
  \preceq_R Q$. Obviously $F \circ (F \circ Q) = F \circ Q$.

  This shows that the map : $\circ_F : (\hat{0},P_{top})_{R,e}
  \rightarrow (\hat{0},P_{top})_{R,e}$ is a closure operator.
  And hence $(\hat{0},P_{top})_{R,e}$ is homotopy equivalent to its
  image (see e.g, \cite[Corollary 10.12]{Bj}).   

  Since $P_{top} \not\in \cstar(F)$ and $F \circ P_{top}  \in \cstar(F)
  \cap (\hat{0},P_{top})_{R,e}$, 
  it also follows that $F \circ Q \preceq_R F \circ P_{top}$ for all 
  $Q \in (\hat{0},P_{top})_{R,e}$. Hence the image of $\circ_F$ has
  a maximal element and hence is contractible.  
\end{proof}

\section{Supertopes}
\label{sec:supertope}

In this section we identify supertopes that are relevant for our
purposes and provide the proof of \ref{thm:supertope}.
We also deduce \ref{cor:crucial}, which is crucial 
for the derivation of \ref{thm:varchenko} from \ref{thm:supertope}. 
Throughout this section we assume \ref{sit:*}.

In order to apply the Quillen fiber in the proof
of \ref{thm:varchenko} we need the following lemma.

\begin{Lemma}
  \label{lem:fiberisst}
  Let $S^+,S^- \subseteq E$ such that $S^+\cap S^- = \emptyset, \,
  S^+\cup S^- \ne E$ and $f \in E \setminus (S^+\cup S^-)$. Let
  $\tT^f$ denote the set of topes of $\lL \setminus f$ and
  $\tT^f_{R\setminus f}$ the corresponding tope poset with base
  polytope $R\setminus \{f\}$. Consider the poset map
  $\pi^{f}:\tT(S^+,S^-) \to \tT^{f}(S^+,S^-)$ given by restriction.
  Let $Q \in \tT^{f}(S^+,S^-)$. Then
  \[(\pi^{f})^{-1}(\tT^f_{\preceq Q})=\tT(Q^+,S^-).\]
\end{Lemma}
\begin{proof}
  Let $\tilde Q \in \tT^f_{\preceq Q} \cap \tT^{f}(S^+,S^-)$.  As
  $R\setminus f$ is all positive, we must have $S^-\subseteq \tilde
  Q^- \subseteq Q^-$ and hence $Q^+ \subseteq \tilde Q^+$ implying
  $(\pi^{f})^{-1}(\tilde Q) \subseteq \tT(Q^+,S^-)$. On the other
  hand, if $\hat Q \in \tT(Q^+,S^-)$, then
  $S^-\subseteq \hat Q^- \subseteq Q^-\cup\{f\},\, S^+\subseteq Q^+
  \subseteq \hat Q^+\cup \{f\}$. As $f \not\in S^+ \cup S^-$ we have that  $\pi^f(\hat Q)$ is well defined and
  $\pi^f(\hat Q) \preceq Q$.
\end{proof}

We need another preparatory lemma for  the proof of \ref{thm:supertope}.

\begin{Lemma}\label{lem:supertope}
  Let $\lL$ be an oriented matroid on $E$ and $\tT$ be
  the set of its topes. Let $E= S^+ \dot \cup S^- \dot \cup S^\ast$ be
  a partition of the ground set into nonempty sets $S^+,S^-$ and $S^*$. 
  If for all $f \in S^\ast$ there exists
  $T^f \in \tT$ such that
  \[T^f_g=\left\{ 
    \begin{array}[h]{rcl}
      + &\text{ if } & g \in S^+\\ 
      - &\text{ if } & g \in S^-\\ 
      - &\text{ if } & g \in S^\ast  \setminus\{f\}\\ 
      + &\text{ if } & g =f, 
    \end{array}\right.\]
  then either there exists a tope $T^{\max{}}  \in \tT$ satisfying
  \[T^{\max{}}_g=\left\{
  \begin{array}[h]{rcl}
    + &\text{ if } & g \in S^+\\ 
    - &\text{ if } & g \in S^-\\ 
    - &\text{ if } & g \in S^\ast,  \end{array}\right.\]
  or there exists a tope $T^{\min{}} \in \tT$ satisfying
  \[T^{\min{}}_g=\left\{
  \begin{array}[h]{rcl}
    + &\text{ if } & g \in S^+\\ 
    - &\text{ if } & g \in S^-\\ 
    + &\text{ if } & g \in S^\ast  \end{array}\right.\]
  or a covector $Y \in \lL$ satisfying
  \[Y_g=\left\{
  \begin{array}[h]{rcl}
    + &\text{ if } & g \in S^+\\ 
    - &\text{ if } & g \in S^-\\ 
    0 &\text{ if } & g \in S^0\\ 
    - &\text{ if } & g \in S^\ast \setminus S^0  \end{array}\right.\]
  for some set $\emptyset \ne S^0\subseteq S^\ast$.

  Hence, for a fixed $R \in \tT$ the subposet $\tT(S^+,S^-)$ of
  $\tT_R$ either has a unique maximal element or it has a unique
  minimal element. In particular, it is contractible.
\end{Lemma}
\begin{proof}
  We proceed by induction on $|S^\ast|$. If $|S^\ast|=1$ the assertion
  is trivial. If $S^\ast=\{f,g\}$, then, either $f$ and $g$ are
  antiparallel and we find a $Y$ as required, or on a shortest path
  from $T^f$ to $T$ we must pass through $T^{\max{}}$ or 
  $T^{\min{}}$.  Assume $|S^\ast|\ge 3$ and let $g \in S^\ast$.
  If there exists some $f \in S^\ast \setminus \{g\}$ such that
  eliminating $g$ between $T^g$ and $T^f$ yields a covector $X^f$ such
  that $X^f_{f}\in \{0,-\}$, then $X \circ T^h$ for $ h \in S^\ast
  \setminus \{f,g\}$ is an element $T^{\max{}}$. Hence we may
  assume that for all $f \in S^\ast \setminus\{g\}$ we find $X$
  satisfying
  \[X^f_h=\left\{
    \begin{array}[h]{rcl}
      + &\text{ if } & h \in S^+\\
      - &\text{ if } & h \in S^-\\ 
      - &\text{ if } & h \in S^\ast  \setminus\{f,g\}\\ 
      + &\text{ if } & h =f\\ 
      0 &\text{ if } & h =g. 
    \end{array}\right.\]
  Then the image of $X^f$ in $\lL/g$ satisfies the assumptions of the lemma in the oriented
  matroid $\lL/g$. By induction we find either an
  appropriate $\tilde Y \in \lL /g$ which clearly yields a $Y$ as
  required in $\lL$, or we find an element $X^{\max{}} \in \lL$ such that  
  \[X^{\max{}}_h=\left\{
  \begin{array}[h]{rcl}
    + &\text{ if } & h \in S^+\\
    - &\text{ if } & h \in S^-\\ 
    - &\text{ if } & h \in S^\ast \setminus\{g\}\\ 
    0 &\text{ if } & h =g 
  \end{array}\right.\] 
  and $X^{\max{}}\circ T^f$ is as required for $f \in S^\ast \setminus
  g$ and similarly $X^{\min{}}\circ T^g$ in the remaining case.

  The assertion about the topology of $\tT(S^+,S^-)$ now follows from the fact
that either  $T^{\max{}}$, $T^{\min{}}$ of $Y \circ (-R)$ has a
  unique minimal or a unique maximal element.

\end{proof}

Now we are in position to prove \ref{thm:supertope}.

\begin{proof}[Proof of \ref{thm:supertope}]
  We proceed by induction on $|E \setminus(S^+ \cup S^-)|$. If $S^+
  \cup S^-=E$, then $\tT(S^+,S^-)$ is a singleton and thus
  contractible. If $S^+ \cup S^-\ne E$, then $S^\ast:=E \setminus (S^+
  \cup S^-)\ne \emptyset$. If for all $f \in S^\ast$ there exists
  $T^f$ as in \ref{lem:supertope}, $\tT(S^+,S^-)$ is contractible by
  \ref{lem:supertope}. Hence we may assume that there exists $f \in
  S^\ast$ such that $T^f \not \in \tT$.  Let $\tT^{f}_{R \setminus
    \{f\}}$ denote the tope poset in the oriented matroid $\lL
  \setminus f$ with base tope $R \setminus \{f\}$.  By inductive
  assumption its subposet $\tT^{f}(S^+,S^-)$ is contractible. Consider
  the poset map $\pi^{f}:\tT(S^+,S^-) \to \tT^{f}(S^+,S^-)$ given by
  restriction. Let $Q \in \tT^{f}(S^+,S^-)$. By \ref{lem:fiberisst} 
  \[(\pi^{f})^{-1}(\tT^f_{\preceq Q})=\tT(Q^+,S^-).\] Clearly $S^+
  \subseteq Q^+$. If $S^+ \subsetneq Q^+$, then
  $(\pi^{f})^{-1}(Q_{\preceq})$ is contractible by inductive
  assumption. Consider the case that $S^+=Q^+$. By the choice of $f$
  the preimage $(\pi^{f})^{-1}(Q)$ is a singleton $\{Z\}$ with $Z=-$. Hence,
  this is the unique maximal element in $(\pi^{f})^{-1}(Q_{\preceq})$ and
  that fiber is also contractible.  Hence by \ref{lem:quillen}
  $\tT(S^+,S^-)$ and $\tT^f(S^+,S^-)$ are homotopy equivalent and the
  claim follows.
\end{proof}


\begin{Corollary}
  \label{cor:crucial}
  Let $R \in \tT$ be the base tope of the poset $\tT_R$. Let 
  $e \not\in S\subseteq E$. Then 
  $$\sum_{\genfrac{}{}{0pt}{}{Q \in \tT(\emptyset,\{e\})}{ S= \Sep(P,Q) \cap \Sep(Q,R)}}  \mu((\hat{0},Q)_{R,e}) = \left\{ \begin{array}{ccc} -1 & \mbox{~if~} & S = \emptyset \\ 
                                       0 & \mbox{~if~} & S \neq \emptyset 
                    \end{array} \right. .$$
\end{Corollary}
\begin{proof}
  We prove the assertion by induction on $\# S$.
 
  If $S = \emptyset$ then 
  \begin{eqnarray*} 
     \sum_{\genfrac{}{}{0pt}{}{Q \in \tT(\emptyset,\{e\})}{ S= \Sep(P,Q) \cap \Sep(Q,R)}}  \mu((\hat{0},Q)_{R,e}) & = & \sum_{\hat{0} <_{R,e} Q \leq_{R,e} P} \mu((\hat{0},Q)_{R,e}) \\
           & = & - \mu((\hat{0},\hat{0})_{R,e}) \\
           & = & -1.
  \end{eqnarray*}

  Assume $\# S > 0$. 
  Set $$T^+ = \{ f\in E \setminus (S \cup \{e\})~|~R_f = +\} \mbox{~and~} 
  T^- = \{ f\in E \setminus S ~|~R_f = -\} \cup \{e\}.$$
  
  Then 
  \begin{eqnarray} 
     \label{eq:supertope}
     \sum_{\genfrac{}{}{0pt}{}{Q \in \tT(\emptyset,\{e\})}{ \Sep(P,Q) \cap \Sep(Q,R) \subseteq S}}  \mu((\hat{0},Q)_{R,e}) & = & 
     \sum_{Q \in \tT(T^+,T^-)} \mu((\hat{0},Q)_{R,e}) 
  \end{eqnarray}

  The right hand side of \eqref{eq:supertope} is the sum of 
  M\"obius function values from $\hat{0}$ to $P$ where 
  $P \neq \hat{0}$ ranges by
  \ref{thm:supertope} over the elements of a contractible poset.
  By classical M\"obius function theory (see e.g.\ \cite[(9.14)]{Bj}) this sum then is 
  $-\mu(\hat{0},\hat{0})=-1$ plus the M\"obius number of the poset. 
  Since the poset is contractible its M\"obius number is $0$ and we
  have shown that:

  \begin{eqnarray} 
     \label{eq:supertope2}
     \sum_{\genfrac{}{}{0pt}{}{Q \in \tT(\emptyset,\{e\})}{ \Sep(P,Q) \cap \Sep(Q,R) \subseteq S}}  \mu((\hat{0},Q)_{R,e}) & = & -1 
  \end{eqnarray}

  Now rewrite the right hand side of \eqref{eq:supertope2} as:

  \begin{eqnarray} 
     \label{eq:supertope3}
     \sum_{\genfrac{}{}{0pt}{}{Q \in \tT(\emptyset,\{e\})}{ \Sep(P,Q) \cap \Sep(Q,R) \subseteq S}}  \mu((\hat{0},Q)_{R,e}) & = & \sum_{T \subseteq S} \sum_{\genfrac{}{}{0pt}{}{Q \in \tT(\emptyset,\{e\})}{ \Sep(P,Q) \cap \Sep(Q,R) = T}}  \mu((\hat{0},Q)_{R,e}) 
  \end{eqnarray}

  By induction the summand 
  $\displaystyle{\sum_{\genfrac{}{}{0pt}{}{Q \in \tT(\emptyset,\{e\})}{ \Sep(P,Q) \cap \Sep(Q,R) = T}}  \mu((\hat{0},Q)_{R,e})}$ 
  is $0$ for $T \neq S,\emptyset$ and $-1$ for $T = \emptyset$.
  Thus combining \eqref{eq:supertope2} and \eqref{eq:supertope3} we obtain:

  \begin{eqnarray*} 
     \label{eq:supertope4}
     -1 & = & \sum_{\genfrac{}{}{0pt}{}{Q \in \tT(\emptyset,\{e\})}{ \Sep(P,Q) \cap \Sep(Q,R) \subseteq S}}  \mu((\hat{0},Q)_{R,e}) \\
        & = & -1+\sum_{\genfrac{}{}{0pt}{}{Q \in \tT(\emptyset,\{e\})}{ \Sep(P,Q) \cap \Sep(Q,R) = S}}  \mu((\hat{0},Q)_{R,e}) 
  \end{eqnarray*}

  From this we conclude $$\sum_{\genfrac{}{}{0pt}{}{Q \in \tT(\emptyset,\{e\})}{ \Sep(P,Q) \cap \Sep(Q,R) = S}}  \mu((\hat{0},Q)_{R,e}) = 0.$$
\end{proof}

\begin{Remark}
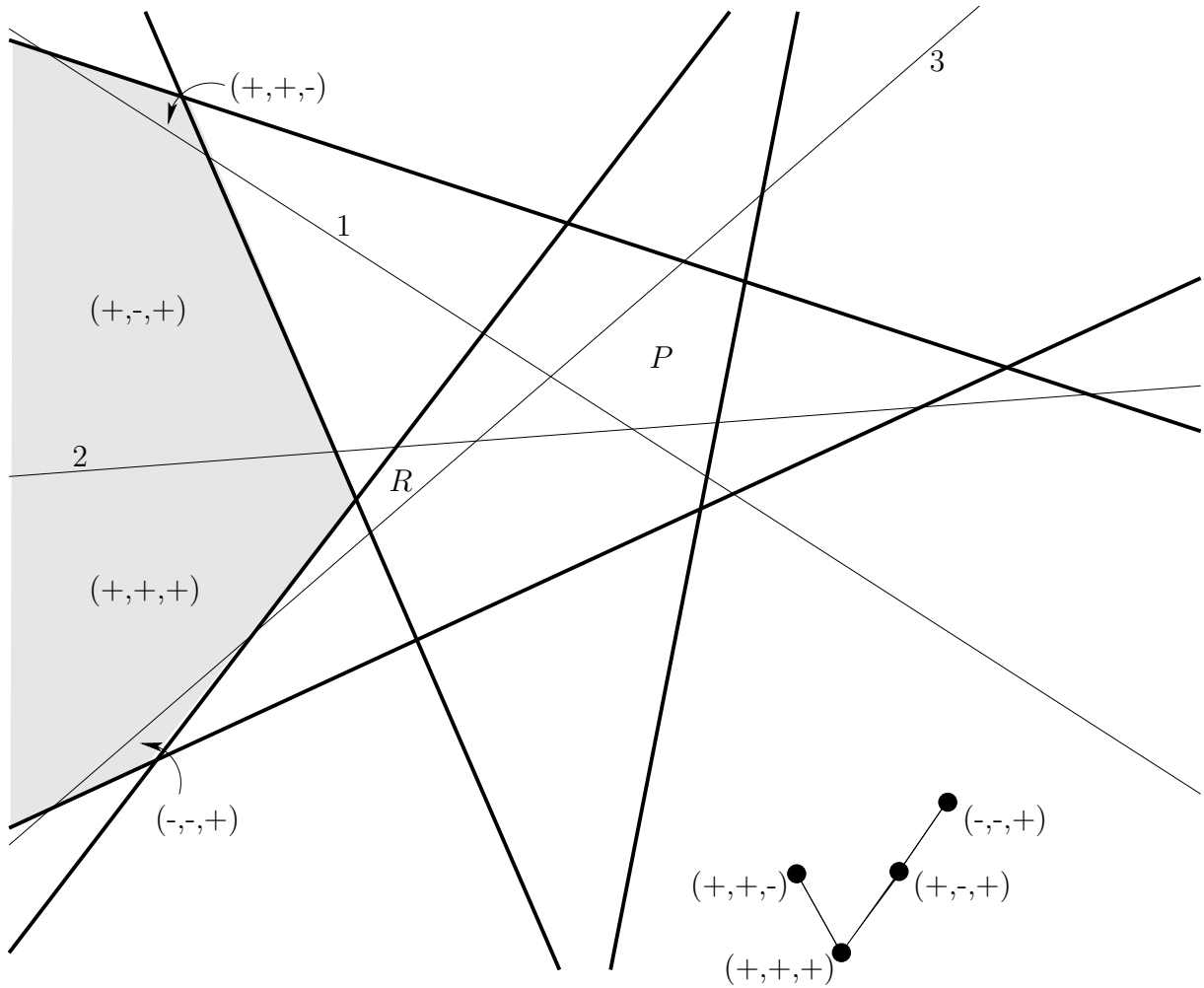
\begin{figure}[htp]
  \centering
  \begin{picture}(200,0)%
    \includegraphics[width=\textwidth]{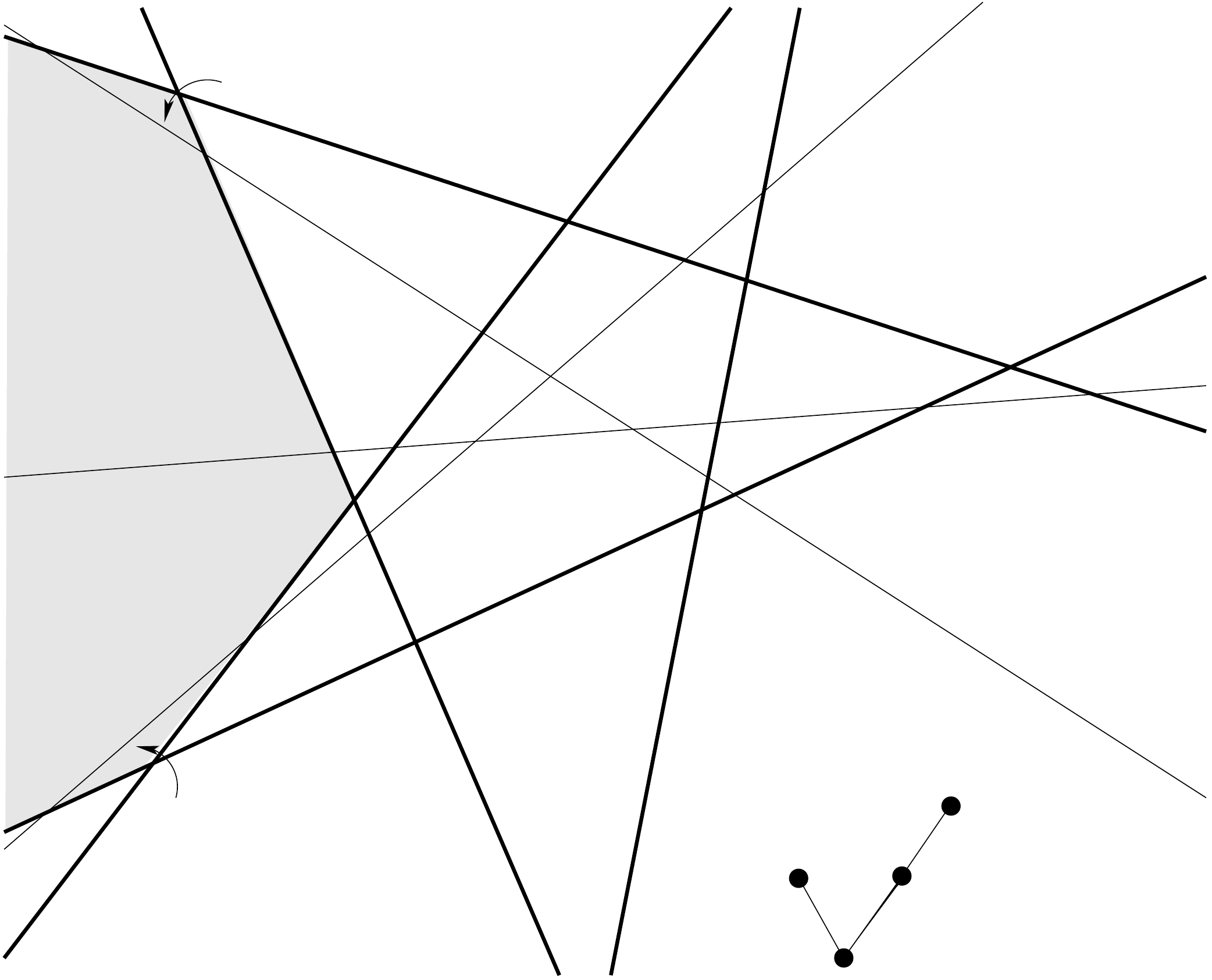}%
  \end{picture}%
  \setlength{\unitlength}{4144sp}%
  \begingroup\makeatletter\ifx\SetFigFont\undefined%
    \gdef\SetFigFont#1#2#3#4#5{%
    \reset@font\fontsize{#1}{#2pt}%
    \fontfamily{#3}\fontseries{#4}\fontshape{#5}%
    \selectfont}%
  \fi\endgroup%
  \begin{picture}(9540,7394)(1654,-8682)
    \put(820,-5750){\makebox(0,0)[lb]{\smash{{\SetFigFont{12}{14.4}{\rmdefault}{\mddefault}{\updefault}{\color[rgb]{0,0,0}$R$}%
}}}}
    \put(2400,-5000){\makebox(0,0)[lb]{\smash{{\SetFigFont{12}{14.4}{\rmdefault}{\mddefault}{\updefault}{\color[rgb]{0,0,0}$P$}%
}}}}
    \put(2850,-8700){\makebox(0,0)[lb]{\smash{{\SetFigFont{12}{14.4}{\rmdefault}{\mddefault}{\updefault}{\color[rgb]{0,0,0}(+,+,+)}%
}}}}
    \put(2650,-8200){\makebox(0,0)[lb]{\smash{{\SetFigFont{12}{14.4}{\rmdefault}{\mddefault}{\updefault}{\color[rgb]{0,0,0}(+,+,-)}%
}}}}
    \put(4000,-8200){\makebox(0,0)[lb]{\smash{{\SetFigFont{12}{14.4}{\rmdefault}{\mddefault}{\updefault}{\color[rgb]{0,0,0}(+,-,+)}%
}}}}
    \put(4300,-7800){\makebox(0,0)[lb]{\smash{{\SetFigFont{12}{14.4}{\rmdefault}{\mddefault}{\updefault}{\color[rgb]{0,0,0}(-,-,+)}%
}}}}
    \put(-150,-3350){\makebox(0,0)[lb]{\smash{{\SetFigFont{12}{14.4}{\rmdefault}{\mddefault}{\updefault}{\color[rgb]{0,0,0}(+,+,-)}%
}}}}
    \put(-1000,-4700){\makebox(0,0)[lb]{\smash{{\SetFigFont{12}{14.4}{\rmdefault}{\mddefault}{\updefault}{\color[rgb]{0,0,0}(+,-,+)}%
}}}}
    \put(-1000,-6400){\makebox(0,0)[lb]{\smash{{\SetFigFont{12}{14.4}{\rmdefault}{\mddefault}{\updefault}{\color[rgb]{0,0,0}(+,+,+)}%
}}}}
    \put(-600,-7800){\makebox(0,0)[lb]{\smash{{\SetFigFont{12}{14.4}{\rmdefault}{\mddefault}{\updefault}{\color[rgb]{0,0,0}(-,-,+)}%
}}}}
    \put(500,-4200){\makebox(0,0)[lb]{\smash{{\SetFigFont{12}{14.4}{\rmdefault}{\mddefault}{\updefault}{\color[rgb]{0,0,0}$1$}%
}}}}
    \put(-1100,-5600){\makebox(0,0)[lb]{\smash{{\SetFigFont{12}{14.4}{\rmdefault}{\mddefault}{\updefault}{\color[rgb]{0,0,0}$2$}%
}}}}
    \put(4100,-3200){\makebox(0,0)[lb]{\smash{{\SetFigFont{12}{14.4}{\rmdefault}{\mddefault}{\updefault}{\color[rgb]{0,0,0}$3$}%
}}}}
  \end{picture}%
  \caption{The shaded region has two maximal elements}
  \label{fig:1}
\end{figure}

Denham and Hanlon mention in \cite{DH} that a ``routine argument
shows'' that the poset on $\{Q \in \tT(\emptyset,\{e\})\mid
\Sep(P,Q) \cap \Sep(Q,R) = S\}$ induced by $\tT_R$ always contains a
unique maximal element. While this can be shown to hold true for line
arrangements, it fails already in 3-dimensional hyperplane
	arrangements. In Figure~\ref{fig:1}, we provide a counterexample.
The
element $e$ is supposed to be the drawing plane. The tope $R$ is below
and $P$ and the shaded region above $e$. The separator $\Sep(P,R)$
without $e$ is given by the thin lines, while the intersection of the
remaining hyperplanes with $e$ are the bold lines. $S$ is given by the
two bold lines that intersect in a vertex at $R$. The poset induced on the
shaded regions is 
sketched on the bottom of the figure. 
While it is contractible it has two
maximal elements.
\end{Remark}
\section{The Varchenko matrix}
\label{sec:varchenko}

In this section we prove \ref{thm:varchenko} and its corollaries.
The proof consists of a factorization of the matrix
$\Var$ into matrices with controllable determinant.

Recall, that we assume $\tT$ to be linearly ordered.
For any sign pattern $\epsilon =
(\epsilon_1,\epsilon_2) \in \{ +,-\}^2$ let $\Var^{e,\epsilon}$
be a $(\ell \times \ell)$-matrix
with rows indexed by $\tT(\{e\} ,\emptyset)$ for $\epsilon_1 =+$,
$\tT(\emptyset, \{e\})$ for $\epsilon_1 = -$ and columns indexed
by $\tT(\{e\},\emptyset)$ for $\epsilon_2 = +$, $\tT(\emptyset, \{e\})$ for
$\epsilon_2 = -$. For a tope $P$ indexing a row and
a tope $Q$ indexing a column we set $\Var^{e,\epsilon}_{P,Q} = 
\Var_{P,Q}$. 
We set $\ell = \# \tT(\{e\},\emptyset) = 
\# \tT(\emptyset,\{e\})$. Note that $\ell = \frac{1}{2} \# \tT$ is
independent of $e$.
We fix a linear ordering on $E$ and set $M^e$ to be
the  $(\ell \times \ell)$-matrix with rows indexed
by $\tT(\emptyset,\{e\})$, columns indexed by $\tT(\{e\},\emptyset)$ and entries
\[M^e_{Q,R} = \left\{
  \begin{array}[h]{cl}
- \mu((\hat{0},Q)_{R,e}) \cdot \Var_{Q,R} & \text{if }e \text{ is the maximal element of }\Sep(Q,R) \text{
 and }\\0& \text{otherwise,}    
  \end{array}\right.\]
 where
$Q \in \tT(\emptyset,\{e\})$ and $R \in \tT(\{e\},\emptyset)$.  
We write $\Id_{\ell}$ for the $(\ell \times \ell)$-identity matrix
and define 
$$\mM^e = 
      \left(  \begin{array}{cc} 
         \Id_{\ell}  & M^e  \\ 
         M^e  & \Id_{\ell} 
       \end{array}
      \right) .$$

\begin{Lemma}
   \label{lem:fac1}
   Let $e$ be the maximal element of $E$. Then 
   $\Var^{e,(-,+)}$ factors as
   \begin{eqnarray}
      \label{eq:tobeproved}
      \Var^{e,(-,+)} & = & \Var^{e,(-,-)} \cdot M^e.
   \end{eqnarray}
\end{Lemma}
\begin{proof}
  For $P \in \tT(\emptyset,\{e\})$ and $R \in \tT(\{e\},\emptyset)$ the entry in row $P$ and
  column $R$ on the left hand side of \eqref{eq:tobeproved} is 
  $\Var_{P,R}$. On the right hand side the corresponding entry is:
  \begin{eqnarray*} 
     \sum_{Q \in \tT(\emptyset,\{e\})} \Var_{P,Q} \cdot M^e_{Q,R} & = & - \sum_{Q \in \tT(\emptyset,\{e\})} \mu((\hat{0},Q)_{R,e}) \cdot  \Var_{P,Q} \cdot \Var_{Q,R} 
  \end{eqnarray*} 

  By definition for $Q \in \tT(\emptyset,\{e\})$ we have 
  \begin{eqnarray*}
    \Var_{P,Q} \cdot \Var_{Q,R} & = & \Var_{P,R} \cdot \prod_{f \in \Sep(P,Q) \cap \Sep(Q,R)} U_f^2 
  \end{eqnarray*}

  Thus the claim of the lemma is proved once we have shown that for a fixed subset $S \subseteq E$ 
  and fixed $Q,R$ we have:
  \begin{eqnarray}
     \label{eq:mueq}
     \sum_{\genfrac{}{}{0pt}{}{Q \in \tT(\emptyset,\{e\})}{ S= \Sep(P,Q) \cap \Sep(Q,R)}}  \mu((\hat{0},Q)_{R,e}) = \left\{ 
                \begin{array}{cc} 0 & \mbox{~if~} S\neq \emptyset \\
                                 -1 & \mbox{~otherwise.} 
                \end{array} \right. .
  \end{eqnarray}

  But this is the content of \ref{cor:crucial} and we are done.
\end{proof}

Next we use the matrices $\mM^e$ to factorize $\Var$. 
The following lemma yields the inductive step in the 
factorization.
 
\begin{Lemma}
  \label{lem:fac}
  Let $e$ be the maximal element of $E$ and
  let $\Var_{U_e = 0}$ be the matrix $\Var$ after
  evaluating $U_e$ to $0$. 
  Then 
  $$\Var = \Var_{U_e = 0} \cdot \mM^e 
  $$
\end{Lemma}
\begin{proof}
   Let $\tT(\emptyset,\{e\}) = \{ P_1,\ldots, P_{\ell} \}$ and $\tT(\{e\},\emptyset) = \{ P_{\ell +1},\ldots, P_{2\ell}\}$ 
   be numbered such that $-P_i = P_{\ell+i}$ for
   $1 \leq i \leq \ell$. 
   Assume the rows and columns of $\Var$ are ordered according to this
   numbering of $\tT$. This yields a block decomposition of $\Var$ as
   $$\Var = \left( 
       \begin{array}{cc} 
         \Var^{e,(-,-)} & \Var^{e,(-,+)} \\ 
         \Var^{e,(+,-)} & \Var^{e,(+,+)} 
       \end{array}
      \right) .
   $$

   Since $\Var_{P,Q} = \Var_{-P,-Q}$ it follows that
   $\Var^{e,(-,-)}= \Var^{e,(+,+)}$ and
   $\Var^{e,(-,+)}= \Var^{e,(+,-)}$.
   Then by \ref{lem:fac1} we know 
   $\Var^{e,(-,+)}  = \Var^{e,(-,-)} \cdot M^e$ and hence
   $$\Var^{e,(+,-)} = \Var^{e,(-,+)} = \Var^{e,(-,-)} \cdot M^e = \Var^{e,(+,+)} \cdot M^e.$$
   Thus 

   \begin{eqnarray}
     \label{eq:ma1}
     \Var & = & \left( 
       \begin{array}{cc} 
         \Var^{e,(-,-)} & 0 \\ 
         0 & \Var^{e,(+,+)} 
       \end{array}
       \right) \cdot 
       \left( 
       \begin{array}{cc} 
         \Id_{\ell}  & M^e  \\ 
         M^e  & \Id_{\ell} 
       \end{array}
       \right) \\
     \nonumber & = &  \left( 
       \begin{array}{cc} 
	 \Var^{e,(-,-)} & 0 \\ 
	 0 & \Var^{e,(+,+)} 
       \end{array}
       \right) \cdot \mM^e .
   \end{eqnarray}

   Now the monomial $\Var_{P,Q}$ has a factor $U_e$ if and only if
   $P \in \tT(\emptyset,\{e\})$ and $Q \in \tT(\{e\},\emptyset)$ or
   $P \in \tT(\{e\},\emptyset)$ and $Q \in \tT(\emptyset,\{e\})$.
   Hence 
   \begin{eqnarray}
     \label{eq:ma2}
     \Var_{U_e = 0}  & = & \left( 
        \begin{array}{cc} 
           \Var^{e,(-,-)} & 0 \\ 
	   0 & \Var^{e,(+,+)} 
	\end{array} 
	\right).
   \end{eqnarray}
   Combining \eqref{eq:ma1} und \eqref{eq:ma2} yields the claim.
\end{proof}
	   
Now we are in position to state and prove the crucial factorization.

\begin{Proposition}
  \label{pro:fac}
  Let $E = \{ e_1 \prec \cdots \prec e_r\}$ be a fixed
  ordering. Then 
  $$\Var = \mM^{e_1} \cdots \mM^{e_r}.$$ 
\end{Proposition}
\begin{proof}
  We will prove by downward induction on $i$ that
  \begin{eqnarray}
     \label{eq:var}
     \Var & = & \Var_{U_i=\cdots = U_r = 0} \cdot \mM^{e_i} \cdots \mM^{e_r}.
  \end{eqnarray}

  For $i=r$ the assertion follows directly from 
  \ref{pro:fac}.
  For the inductive step assume
  $i > 1$ and \eqref{eq:var} holds for $i$. 
  We know from \ref{lem:fac} that if we choose a linear ordering
  on $E$ for which $e_{i-1}$ is the largest element then 
  \begin{eqnarray}
	    \label{eq:ind1} 
	    \Var & = & \Var_{U_{i-1}=0}\cdot \nN,
	  \end{eqnarray}
	  where
	  $\nN = (N_{Q,R})_{Q,R \in \tT}$ is defined as
	  $$N_{Q,R} = \left\{ 
	    \begin{array}{ccc} 
		1 & \mbox{~if~} & Q = R \\ 
	       -\mu((\hat{0},Q)_{R,e_{i-1}})\, \Var_{Q,R} & \mbox{~if~} & R \in \tT(\{e_{i-1}\},\emptyset), Q \in \tT(\emptyset,\{e_{i-1}\})  \\
	       -\mu((\hat{0},-Q)_{-R,e_{i-1}})\, \Var_{Q,R} & \mbox{~if~} & -R \in \tT(\{e_{i-1}\},\emptyset), -Q \in \tT(\emptyset,\{e_{i-1}\}) \\

		0 & \mbox{~otherwise} & 
	    \end{array} \right. . 
	  $$
	  This shows:
	  $$(N_{Q,R})_{U_i=\cdots   = U_r = 0} \left\{ 
	    {\small \begin{array}{cl} 
		1 & \mbox{~if~}  Q = R \\ 
	       -\mu((\hat{0},Q)_{R,e_{i-1}})\, \Var_{Q,R} & \mbox{~if~} e_{i-1} \mbox{~is the largest element in~} \Sep(Q,R),  \\ 
                  &              R \in \tT(\{e_{i-1}\},\emptyset), Q \in \tT(\emptyset, \{e_{i-1}\})  \\
	       -\mu((\hat{0},-Q)_{-R,e_{i-1}})\, \Var_{Q,R} & \mbox{~if~} e_{i-1} \mbox{~is the largest element in~} \Sep(Q,R), \\
                  &             -R \in \tT(\{e_{i-1}\},\emptyset), -Q \in \tT(\emptyset,\{e_{i-1}\})  \\
		0 & \mbox{~otherwise}  
	    \end{array}} \right. . 
	  $$
	  But then
	  $\nN_{U_i=\cdots   = U_r = 0} = \mM^{e_{i-1}}$. 
	 
	  Now \eqref{eq:ind1} implies
	  \begin{eqnarray*}
	      \Var_{U_i=\cdots = U_{r}=0} &= & \Var_{U_{i-1}= \cdots = U_{r} = 0} \cdot
	       \nN_{U_i = \cdots = U_r = 0} \\
					  & = & \Var_{U_{i-1}= \cdots = U_{r} = 0} \cdot
	\mM^{e_{i-1}} 
	  \end{eqnarray*}

	  With the induction hypothesis this completes the induction step by
	  \begin{eqnarray*}
	     \Var & = & \Var_{U_i=\cdots = U_r = 0} \cdot \mM^{e_i} \cdots \mM^{e_r} \\
		  & = & \Var_{U_{i-1} = U_i=\cdots = U_r = 0} \cdot \mM^{e_{i-1}} \cdots \mM^{e_r}. 
	  \end{eqnarray*}

	  For $i =1$ the matrix $\Var_{U_1=\cdots = U_r = 0}$ is the identity matrix. 
	  Thus \eqref{eq:var} yields:
	  $$\Var = \mM^{e_1} \cdots \mM^{e_r}.$$
	\end{proof}

Let $F \in \lL$ and $e \in z(F)$ be the maximal element of
$z(F)$. Define $\tT^{F,e}$ as the set of topes $P \in \tT$
such that $F$ is the maximal element of $\lL$ for which $F_e =
0$ and $F \leq P$.

\begin{Proposition}\label{pro:contraction}
  For any pair of topes  $Q,R \in \tT^{F,e}$ we have 
  \begin{eqnarray*} 
     \mu((\hat{0},\pm Q)_{\pm R,e}) = 
       \left\{ \begin{array}{ccc} 
          -(-1)^{\rank(\lL|_{z(F)}} & \mbox{~if~} & Q_{z(F)} = -R_{z(F)} \\ 
          0 & \mbox{~otherwise~} & 
  \end{array} \right. .  
  \end{eqnarray*} 
  
\end{Proposition}

\begin{proof}
  By the definition of $\tT^{F,e}$ we have $F \le Q,R$. Thus, if we
  consider the poset $\tT_{R_{z(F)},e}$ in the contraction $\lL/z(F)$
  we find that the interval $(\hat{0},\pm Q)_{\pm R,e}$ is isomorphic
  to $(\hat{0},\pm Q_{z(F)})_{\pm R_{z(F)},e}$. Furthermore, since $F$
  is the maximal element satisfying $F_e=0$ and $F \le Q$, $e$ does
  not define a proper face of $Q_{z(F)}$. Hence the claim follows from
\ref{co:moebius}.
\end{proof}
We define $b_{F,e} = 0$ if $e$ is not the maximal element of 
$z(F)$ and $\frac{1}{2} \# \tT^{F,e}$ otherwise. 
Since $P \mapsto F \circ (-P)$ is a perfect pairing on 
$\tT^{F,e}$ it follows that $\tT^{F,e}$ contains an even number of
topes. In particular, $b_{F,e}$ is an integer. 
We denote by $\mM^{F,e}$ the submatrix of $\mM^e$ obtained by selecting rows and columns indexed by $\tT^{F,e}$.  

\begin{Lemma}
  \label{lem:formula}
  Let $F \in \lL$ and $e \in z(F)$.
  If $\tT^{F,e} \neq \emptyset$. 
  then 
  $$\det(\mM^{F,e}) = (1-a(F)^2)^{b_{F,e}}.$$
\end{Lemma}
\begin{proof}
  By definition of $\mM^e$ we obtain that for $Q,R \in \tT^{F,e}$ we
  have
  $$
    \mM^e_{Q,R} = 
       \left\{ 
          \begin{array}{ccc}
             1 & \mbox{~if~} & Q = R \\
             -\mu((\hat{0},Q)_{R,e}) \cdot \Var_{Q,R} & \mbox{~if~} & 
		 e \mbox{~is the largest element of~} \Sep(Q,R), \\
               &             & R \in \tT(\{e\},\emptyset),  Q \in \tT(\emptyset,\{e\})  \\
             -\mu((\hat{0},-Q)_{-R,e}) \cdot \Var_{Q,R} & \mbox{~if~} & 
		 e \mbox{~is the largest element of~} \Sep(Q,R), \\
               &             &-R \in \tT(\{e\},\emptyset), -Q \in \tT(\emptyset,\{e\})  \\
		     0 & \mbox{~otherwise} & 
		  \end{array}
	       \right. .
	  $$

  If $Q_{z(F)} = -R_{z(F)}$ then $\Var_{Q,R} = a(F)$.
  Using~\ref{pro:contraction} we find
  \begin{eqnarray*} 
    \mM^e_{Q,R} = 
       \left\{ 
          \begin{array}{ccc} 
                             1 & \mbox{~if~} & Q = R \\ 
             -(-1)^{\rank(\lL|_S)} a(F) & \mbox{~if~} & Q = F \circ (-R) \\
                                       &             & e \mbox{ largest element of } \Sep(Q,R) \\
                             0 & \mbox{~otherwise~} & 
          \end{array} \right. .
  \end{eqnarray*} 
  We order rows and columns of $\mM^{F,e}$ so that the elements 
  $R$ and $F \circ (-R)$ are  
  paired in consecutive rows and columns.
  With this ordering $\mM^{F,e}$ is a block diagonal
  matrix having along its diagonal $b_{F,e}$ two by two matrices 
  $$\left( \begin{array}{cc}
        1                         & -(-1)^{\rank(\lL|_{z(F)})} a(F) \\
        -(-1)^{\rank(\lL|_{z(F)})} a(F) & 1 
           \end{array} \right)
  $$
  if $e$ is the maximal element of $z(F)$ and 
  identity matrices otherwise.  In any case we find
  $\det(\mM^{F,e}) = (1-a(F)^2)^{b_{F,e}}$ as desired.
\end{proof}
    
\begin{Lemma}
  \label{lem:block}
  After suitably ordering $\tT$ the matrix $\mM^e$ is the block 
  lower triangular matrix
  with the matrices $\mM^{F,e}$ for $F \in \lL$ with $F_e = 0$ and
  $\tT^{F,e} \neq \emptyset$ on the main
  diagonal.
\end{Lemma}
\begin{proof}
  We fix a linear ordering of $\tT$ such that for fixed $e \in E$ and $F \in \lL$ the 
  topes from  $\tT^{F,e}$ form an interval 
  and such that the topes from $\tT^{F,e}$ precede those of
  $\tT^{F',e}$ if $F < F'$.

  For this order the claim follows if we show that the entry $(\mM^e)_{Q,R}$ is zero 
  whenever $R \in \tT^{F,e}$, $Q \in \tT^{F',e}$ and $F' < F$.

  If $Q_e = R_e$ then by $Q \neq R$ we have 
  $(\mM^e)_{Q,R} = 0$. Hence it suffices to consider the case $Q_e \neq R_e$.  
 
 If $Q \not\in \cstar(F),\, Q \in \tT(\emptyset,\{e\})$ and $R\in \tT(\{e\},\emptyset)$ then 
  it follows from \ref{thm:moebius} that 
  $\mu((\hat{0},Q)_{R,e}) = 0$ and therefore $(\mM^e)_{Q,R} = 0$. 
  Analogously if $Q \not\in \cstar(F),\,-Q \in \tT(\emptyset,\{e\})$ and $-R\in \tT(\{e\},\emptyset)$ then 
  $\mu((\hat{0},-Q)_{-R,e}) = 0$ and therefore $(\mM^e)_{Q,R} = 0$. 

  On the other hand, if $Q \in \cstar(F)$, then in particular $F \leq
  Q$. Since by definition of $\tT^{F',e}$ we have that $F'$ is the
  maximal covector such that $F' \leq Q$ and $F'_e=0$ it follows that
  $F \leq F'$.  Since $F \neq F'$ we must have that $F < F'$, i.e.\
  $(\mM^e)_{Q,R}$ is an entry above the diagonal and we are done.

\end{proof}

\begin{proof}[Proof of \ref{thm:varchenko}]
  After fixing a linear order on $E$ it follows from \ref{pro:fac} 
  that $\det \Var$ is 
  the product of the determinants of $\mM^e$ for $e \in E$.  
  By \ref{lem:block} the determinant of each $\mM^e$ is
  a product of determinants of $\mM^{F,e}$ for $e \in E$ and $F \in \lL$
  for which $\tT^{F,e} \neq \emptyset$. Then \ref{lem:formula} completes
  the proof.
\end{proof}

As an immediate consequence of the proof we can give a refined version
of \ref{thm:varchenko} which also implies that $\det(\Var)$ only depends
on the matroid $\Mat(\lL)$ underlying the oriented matroid $\lL$.

Let us first state an additional fact about an oriented matroid $\lL$
over ground set $E$ and its underlying matroid $\Mat(\lL)$. 
Recall, that for a fixed $e \in E$ the bounded topes
  of the affine oriented matroid defined by $e$ in $\lL$ are the topes
  in $\lL$ for which $e$ does not define a proper face. By
  \cite[Theorem 4.6.5]{thebook} the cardinality of this set of topes
  is given by twice the $\beta$-invariant
  $\beta(\Mat(\lL))$ of $\Mat(\lL)$ and hence is independent of $e$.

\begin{Corollary}
   \label{cor:varchenko}
   Let $\Var$ be the Varchenko matrix of the oriented matroid with covector set $\lL$ and $M = \Mat(\lL)$ the matroid 
   underlying $\lL$.
   Then
   \begin{eqnarray*}
     \det (\Var) & = & \prod_{A \subseteq E \atop A \text{ is closed in } M} (1-\prod_{e \in A} U_e^2)^{m_A}, 
   \end{eqnarray*}
   where $m_A$ is the product of number of topes in the
   contraction $\lL / A$ and of 
   $\beta(\Mat(\lL|_A))$. In particular, 
   it follows that $\det (\Var))$ only depends on the matroid $M = \Mat(\lL)$.
\end{Corollary}
\begin{proof}
  Fix $A \subseteq E$. Using the notation of \ref{thm:varchenko} it follows
  that $m_A = \sum_{\genfrac{}{}{0pt}{}{F \in \lL}{z(F) = A}} b_F$.
  The number of summands equals the number of topes in the contraction
  $\lL/A$ of $A$. The latter only depends on $\Mat(\lL/A)$ which only depends on 
  $\Mat(\lL)$. By \ref{lem:formula} we have $b_F = b_{F,e}$ where
  $e$ is the maximal element of $z(F)$. Now $b_{F,e}$ is half the
  number of elements of $\tT^{F,e}$, which is the set of topes $P$ 
  for which $F$
  is the unique maximal element of $\lL$ such that $F \leq P$ and $F_e = 0$.
  The map sending $P$ to $P_{z(F)}$ is then a bijection between the
  topes in $\tT^{F,e}$ and the topes of $\lL|_{z(F)}$
  for which $e$ does not define a proper face. 
  As mentioned above, by \cite[Theorem 4.6.5]{thebook} the number
  of topes in $\lL|_{z(F)}$ for which $e \in z(F)$ does not define a proper 
	face in $\lL|_{z(F)}$ is independent of $e$ and coincides with twice
  the beta invariant $\beta(\Mat(\lL|_{z(F)}))$.
  Hence we find that $b_{F,e} =
  \beta(\Mat(\lL|_{z(F)}))$.
  Since for a nonempty subset the matroid $\Mat(\lL|_A)$ depends on $A$ and $\Mat(\lL)$ only 
  it follows that $m_A$ is an invariant of $\Mat(\lL)$. 
\end{proof}

Finally, as a second corollary we extend \ref{thm:varchenko} to row
and column selected submatrices of $\Var$ corresponding to topes in a
closed supertope. 
For oriented matroids coming from hyperplane arrangements this
formula can also be found in \cite{AM} and \cite{G}.

\begin{Corollary}
   \label{cor:conevarchenko}
   Let $\Var$ be the Varchenko matrix of the oriented matroid with
   covector set $\lL$.  For a subset $E' \subseteq E$ and signs
   $\epsilon=(\epsilon_e)_{e \in E'} \in \{+,-\}^{E'}$ such that
   $\tT(\epsilon^+,\epsilon^-)$ is a closed supertope 
   let $\Var_{\epsilon}$ be the
   matrix constructed from $\Var$ by selecting all rows and columns
   corresponding to topes $P \in \tT$ for which $P_e = \epsilon_e$ for
   $e \in E'$.
 
   Then
   \begin{eqnarray*}
     \det (\Var_\epsilon) & = & 
\prod_{\genfrac{}{}{0pt}{}{F \in \lL}{F_e \neq 0, e \in E'}} (1-a(F)^2)^{b_{F,\epsilon}}, 
   \end{eqnarray*}
   for some numbers $b_{F,\epsilon}$.
\end{Corollary}
\begin{proof}
  Consider the case $E' = \{e\}$. Order  
  $\tT = \{ P_1 \prec \cdots \prec P_{2s}\}$ such that 
  $(P_{1})_e = \cdots = (P_s)_e = +$ and $P_{i+s} = -P_i$ for
  $1 \leq i \leq s$. Then $\Var_\epsilon$ is a block diagonal matrix
  with two blocks identical to $\Var_\epsilon$ on the main diagonal. By \ref{thm:varchenko} we have
  \begin{eqnarray*}
     \det \Var_{U_e = 0} &  = & 
          \prod_{\genfrac{}{}{0pt}{}{F \in \lL}{F_e \neq 0}} (1-a(F))^{b_F} \\
                         & = & \det(\Var_{\epsilon})^2.
  \end{eqnarray*}
  Since the main diagonal in $\Var_{\epsilon}$ is constant $1$
  and since this are the only constant entries it follows that
  $\det(\Var_{\epsilon})$ has constant term $+1$. It follows that
  \begin{eqnarray*} 
    \det(\Var_{\epsilon}) & = & \prod_{\genfrac{}{}{0pt}{}{F \in \lL}{F_e \neq 0}} (1-a(F))^{\frac{b_F}{2}}. 
  \end{eqnarray*}
  Now induction on the cardinality of $E'$ proves the assertion.
\end{proof}

\begin{Remark}
  If $E'=\{e\}$ in \ref{cor:conevarchenko}, i.e.\ in the case of an affine oriented matroid, then \cite[Theorem 4.6.5]{thebook} implies, that  $\det(\Var_{\epsilon})$ is still a matroid invariant. 
\end{Remark}

\begin{Remark}
  The formulas in \ref{thm:varchenko} and
  \ref{cor:varchenko} are very explicit in terms of combinatorial
  invariants of the matroid and are useful for further analysis of the
  matrix. Nevertheless, it seems computationally hard to write down
  the formulas in concrete cases. 
  We refer to \cite{PR} for an analysis of reflection arrangements of
  the symmetric groups. The case that originally motivated 
  Varchenko's work. 
\end{Remark}

  \begin{Remark}
    Recently, Bandelt et.\ al.\ introduced complexes of oriented
    matroids (COMs) \cite{BCK} as families of signed vectors which
    satisfy the covector elimination axiom and a symmetrized version
    of the axiom of conformal composition of oriented matroid theory.
    A COM is an oriented matroid if and only if it contains $0$.
    Examples of COMs are affine oriented matroids and ``closed
    supertopes without boundary''. The authors even conjecture that
    the latter case characterizes COMs. A generalization of
    \ref{cor:conevarchenko} to COMs would support that conjecture. Our
    proof of supertope contractability uses only covector elimination,
    thus \ref{thm:supertope}, \ref{cor:crucial} and \ref{pro:fac}
    should generalize to COMs. In order to zero out the M\"obius
    function values below the diagonal of the $\mM$ matrices though,
    we frequently use results about the global topology of an oriented
    matroid. An exception is \ref{pro:contraction} which should go for
    COMs. Note that its proof is based on properties of a proper
    contraction of a covector. As such a proper contraction of a COM
    contains 0 it is an oriented matroid.  It is not immediately
    clear, though, that this suffices to generalize
    \ref{cor:conevarchenko} without extending all topological results
    to COMs which should be a non-trivial project on its own.
  \end{Remark}

\section*{Acknowledgement}
The authors thank the referee for providing suggestions that
helped to improve the exposition. Moreover, we are grateful for
pointing us to COMs as a possible direction for generalizations.


\begin{thebibliography}{xxx}
  \bibitem{AM}
      M. Aguiar, S. Mahajan, 
      Topics in hyperplane arrangements, 
      Mathematical Surveys and Monographs {\bf 226},
      American Mathematical Society, Providence, RI, 2017. 
  \bibitem{BCK}
     H.-J. Bandelt, V. Chepoi , K. Knauer,  
     COMs: complexes of oriented matroids,
     J. Combin. Theory Ser. A, {\bf 156} (2018) 195--237. 
  \bibitem{Bj}
     A. Bj\"orner, Topological methods, in: Handbook of combinatorics, Volume 2, 
     1819--1872, Elsevier Sci. B. V., Amsterdam, 1995. 
  \bibitem{thebook}
      A. Bj\"orner, M. Las Vergnas, B. Sturmfels, W. White, G.M. Ziegler,
     Oriented matroids. Second edition. Encyclopedia of Mathematics and its
     Applications {\bf 46}. Cambridge University Press, Cambridge, 1999.
  \bibitem{BV}
     T. Bry\l awski, A. Varchenko, 
     The determinant formula for a matroid bilinear form, 
     Adv. Math. {\bf 129} (1997) 1--24. 
  \bibitem{DH}
      G. Denham, P. Hanlon,
      Some algebraic properties of the Schechtman-Varchenko bilinear forms, 
      in: New perspectives in algebraic combinatorics (Berkeley, CA, 1996–97), 
      149--176, Math. Sci. Res. Inst. Publ., {\bf 38}, Cambridge University Press, Cambridge, 
      1999.
  \bibitem{G} 
      R. Gente, 
      The Varchenko Matrix for Cones, 
      PhD-Thesis, Philipps-Universit\"at Marburg, 2013.
  \bibitem{HS} 
      P. Hanlon, R. P. Stanley, 
      A q-deformation of a trivial symmetric group action, 
      Trans. Amer. Math. Soc. {\bf 350} (1998) 4445--4459.
  \bibitem{PR}
      G. Pfeiffer, H. Randriamaro, 
      The Varchenko determinant of a Coxeter arrangement,
      J. Group Theory {\bf 21} (2018) 651--665
  \bibitem{SV1} 
      V.V. Schechtman, A.N. Varchenko, 
      Arrangements of hyperplanes and Lie algebra homology,
      Invent. Math. {\bf 106} (1991) 139--194. 
  \bibitem{SV2} 
      V.V. Schechtman, A.N. Varchenko, 
      Quantum groups and homology of local systems, pp. 182--197, In: 
      Algebraic geometry and analytic geometry (Tokyo, 1990), 
      edited by A. Fujiki et al., Springer, Tokyo, 1991.
  \bibitem{St}
      R.P. Stanley,
      Enumerative combinatorics. Volume 1. 
      Second edition. Cambridge Studies in Advanced Mathematics {\bf 49},
      Cambridge University Press, Cambridge, 2012. 
  \bibitem{V} 
     A. Varchenko, 
     Bilinear form of real configuration of hyperplanes,
     Adv. Math. {\bf 97} (1993) 110--144. 
  \bibitem{Ve1}
     N. Vembar, An Oriented Matroid Determinant,
     Master Thesis, University of North Carolina, Chapell Hill, 2001.
  \bibitem{Ve2}
     N. Vembar, 
     Oriented Matroid Integer Chains, 
     PhD-Thesis, University of North Carolina, Chapell Hill, 2003.
  \bibitem{Z}
     D. Zagier, 
     Realizability of a model in infinite statistics.
     Comm. Math. Phys. {\bf 147} (1992) 199--210. 
\end{thebibliography}
\end{document}